\documentclass{amsart}

\usepackage{etex}
\usepackage{amsmath, amssymb}
\usepackage{array}
\usepackage[frame,cmtip,arrow,matrix,line,graph,curve]{xy}
\usepackage{graphpap, color, paralist, pstricks}
\usepackage[mathscr]{eucal}
\usepackage[pdftex]{graphicx}
\usepackage[pdftex,colorlinks,backref=page,citecolor=blue]{hyperref}
\usepackage{pifont}
\usepackage{cancel}
\usepackage{mathtools}
\usepackage{tikz}

\newtheorem{thm}{Theorem}[section]
\newtheorem{prop}[thm]{Proposition}

\newtheorem{lemma}[thm]{Lemma}
\newtheorem{conjecture}[thm]{Conjecture}
\theoremstyle{definition}

\newtheorem{claim}[thm]{Claim}
\newcommand{\cO}{\mathcal{O} }

\newcommand{\cC}{\mathcal{C} }

\newcommand{\cE}{\mathcal{E} }
\newcommand{\cF}{\mathcal{F} }
\newcommand{\cG}{\mathcal{G} }
\newcommand{\cH}{\mathcal{H} }

\newcommand{\cW}{\mathcal{W} }

\newcommand{\pr}{\mathbb{P}}

\newcommand{\beq}[1]{\begin{equation}\label{#1}}
\newcommand{\enq}[0]{\end{equation}}

\newcommand{\bn}[0]{\bigskip\noindent}
\newcommand{\mn}[0]{\medskip\noindent}
\newcommand{\nin}[0]{\noindent}

\newcommand{\sub}[0]{\subseteq}
\newcommand{\sm}[0]{\setminus}
\renewcommand{\dots}[0]{,\ldots,}

\newcommand{\cee}[0]{{\mathcal C}}

\newcommand{\eee}[0]{{\mathcal E}}
\newcommand{\f}[0]{{\mathcal F}}

\newcommand{\g}[0]{{\mathcal G}}

\newcommand{\Oh}[0]{{\mathcal O}}

\newcommand{\R}[0]{{\mathcal R}}

\newcommand{\T}[0]{{\mathcal T}}
\newcommand{\U}[0]{{\mathcal U}}
\newcommand{\V}[0]{{\mathcal V}}
\newcommand{\W}[0]{{\mathcal W}}

\newcommand{\ra}[0]{\rightarrow}

\newcommand{\bg}[0]{{\bf g}}

\newcommand{\0}[0]{\emptyset}

\newcommand{\C}[2]{{{#1}\choose{{#2}}}}
\newcommand{\Cc}[0]{\tbinom}
\newcommand{\ga}[0]{\alpha }

\newcommand{\go}[0]{\omega}
\newcommand{\gO}[0]{\Omega}

\newcommand{\gz}[0]{\zeta}

\newcommand{\vp}[0]{\varphi}

\newcommand{\comments}[1]{}

\begin{document}

\title{The number of 4-colorings of the Hamming cube}

\author{Jeff Kahn \and Jinyoung Park}
\thanks{The authors are supported by the National Science Foundation under Grant Award DMS1501962}
\thanks{JK is supported by a Simons Fellowship and BSF grant 2014290.}
\thanks{JP is partially supported by CNS-1526333.}
\email{jkahn@math.rutgers.edu, jp1324@math.rutgers.edu}
\address{Department of Mathematics, Rutgers University \\
Hill Center for the Mathematical Sciences \\
110 Frelinghuysen Rd.\\
Piscataway, NJ 08854-8019, USA}

\begin{abstract}
Let $Q_d$ be the $d$-dimensional hypercube and $N=2^d$.  
We prove that
the number of (proper) 4-colorings of $Q_d$ is asymptotically
\[6e2^N,\]
as was conjectured by Engbers and Galvin in 2012. 
The proof uses a combination of information theory (entropy) and 
isoperimetric ideas originating in work of Sapozhenko in the 1980's.
\end{abstract}

\maketitle


\section{Introduction}\label{Intro}

\subsection{Theorem and background}

Write $Q_d$ for the $d$-dimensional hypercube.
(A few basic definitions are given at the beginning of Section~\ref{sec:mp}.)
We use $C_q(G)$ for the number of $q$-colorings of a graph $G$
(where, here and throughout, \emph{coloring} means \emph{proper vertex} coloring),
and $N$ for $2^d$.

The purpose of this paper is to prove the following statement, which was conjectured by Engbers and Galvin
\cite{EG}.
\begin{thm}\label{mainthm}
$C_4(Q_d) \sim 6e2^N $
\end{thm}
\nin
(where $a\sim b$ means $a/b\ra 1$).
We will say why this is natural in a moment.

\mn
\emph{A little background}

\mn

For $q=3$ the result corresponding to Theorem~\ref{mainthm} was proved
by Galvin about fifteen years ago \cite{G}:
\begin{thm}\label{thmG}
$C_3(Q_d) \sim 6e2^{N/2}.$
\end{thm}
\nin
Galvin's work was strongly influenced by the pioneering ideas of
Sapozhenko, used in his proof \cite{Sap87} of the
following result of Korshunov and Sapozhenko \cite{KS},
in which $i(G)$ is
number of independent sets in $G$.
\begin{thm}\label{thmS}
$i(Q_d) \sim 2\sqrt{e}2^{N/2}.$
\end{thm}
\nin
See also the exposition
of Sapozhenko's proof in \cite{GS}, which is the version we will be
referring to in what follows.
The results of both \cite{Sap87} and \cite{G} will again be important
here, though, unlike in \cite{G}, the parts of our argument where these appear will
simply apply the earlier results without further developing the machinery of \cite{Sap87}.

\mn
\emph{Meaning and task}

\mn

In each of the preceding theorems the asymptotic value is an obvious
\emph{lower} bound.  We just say (quickly) how this goes for Theorem~\ref{mainthm}.
Let $\{1,2,3,4\}$ be our set of colors, $\cC$ the set of (six) ordered equipartitions of
this set, and $\f$ the set of (proper) 4-colorings of $Q_d$.
Say $f \in \cF$ \textit{agrees} with $(C,D)\in \cee$ at $v$ if
\[ f_v \in C \Leftrightarrow v \in \cE.\]
For given $(C,D)$ and a fixed $k$, the number of colorings that \emph{dis}agree with $(C,D)$ at
precisely $k$ vertices is asymptotically $\C{N}{k}2^{N-dk}\sim 2^N/k!$
(the $k$ exceptional vertices---\emph{flaws}---will typically have disjoint neighborhoods,
whose colors are determined by those of the flaws), and summing over choices of $(C,D)$ and $k$ gives the value
in Theorem~\ref{mainthm}.

(See also the more general discussion in \cite[Sec.\ 6.1]{EG}; in particular the case $q=4$ of
their Conjecture~6.2 (recalled here in Conjecture~\ref{EGConj}) is our Theorem~\ref{mainthm}.
For $q\geq 5$, colorings will typically have \emph{many} flaws,
and the conjectured asymptotics are for $\log C_q(Q_d)$ rather than $C_q(Q_d)$ itself;
but see Conjecture~\ref{34Conj} below.)

\mn

In fact for almost every $f$ there is some $(C,D)\in \cee$ with which $f$ agrees on all but a
tiny fraction of the vertices;
this special case of Theorem 1.1 of \cite{EG}
is our point of departure:

\begin{thm}\label{start}
There is a fixed $\alpha <2$ such that for all but $|\cF|2^{-\gO(d)}$ $f$'s
in $\cF$ there is some $(C,D) \in \cC$ such that
\[ |\{ v \in V : f \mbox{ disagrees with } (C,D) \mbox{ at } v \}| < \alpha^d.\]
\end{thm}
\nin
For $f$ and $(C,D)$ as in Theorem \ref{start}, call $(C,D)$ the \textit{main phase} of $f$. (So not every $f$ has a main phase, but the number that do not is negligible.) We then write $X_f$ for the set of vertices that disagree with $(C,D)$ at $f$ and call such vertices \textit{bad} (for $f$).

Say a coloring (with a main phase)
is \textit{ideal} if any two of its bad vertices are at distance at least 3.
The preceding lower bound discussion extends to say that the number of ideal colorings is
less than $6\sum_k\C{N}{k}2^{N-dk}< 6e2^N$.
So the asymptotic number of ideal colorings is $6e2^N$ and for Theorem~\ref{mainthm} we should show that the
number of  \textit{non-ideal} colorings is $o(2^N)$, which
in view of Theorem \ref{start} will follow if we show that
the number of \emph{non}-ideal colorings with a given main phase $(C,D)$ is $o(2^N)$.
(In fact it is $2^{N-\gO(d)}$---see following \eqref{calc}---which with
Theorem \ref{start} gives a similar value for what's lost in the ``$\sim$" of Theorem~\ref{mainthm}.)

We may specialize a little further:
Let $\f^*$ be the set of non-ideal $f$'s having main phase $(\{1,2\},\{3,4\})$ and satisfying
\beq{f*def}
|N(X_f \cap \cE)| \ge |N(X_f \cap \cO)|
\enq
(where $N$ is neighborhood).
Then Theorem \ref{mainthm} will follow from
\beq{|f*|}
|\cF^*|=o(2^N),
\enq
and the rest of the paper is concerned with proving this.

\mn

The actual proof is carried out in Section~\ref{Proof}.  Section~\ref{sec:mp} fills in usage notes,
then states our main lemma and shows that it implies \eqref{|f*|},
and Section~\ref{sec:tools} recalls relevant machinery,
consisting mainly of the results of Sapozhenko and Galvin mentioned above and (Shannon) entropy.
Finally, Section ~\ref{More} returns to the conjecture of Engbers and Galvin and suggests that in a couple cases something 
stronger might hold.

The tools named in the preceding paragraph 
are not unexpected, as both have been important in earlier work on questions of the present type,
but the way they are combined here seems interesting.
Specifically, what's perhaps most interesting is the use of entropy
\emph{following} application of \cite{Sap87,G}
(see Section~\ref{Colors}).
This is in contrast to, e.g., the use in \cite{G} of an
entropy-based result from \cite{BHM} as a sort of
preprocessing step (echoed in the role of the entropy-based
Theorem~\ref{thmG} here).
Something similar in spirit---though not in implementation---to what
we do appears
in a recent breakthrough of Peled and Spinka \cite{PS}
(on colorings of $\mathbb{Z}^d$ and related statistical physics models),
which partly inspired our approach.

\section{Setting up}\label{sec:mp}

\subsection{Definitions and such}\label{Defs}

As usual, $[n]$ is shorthand for $\{1\dots n\}$.
Recall that the $d$-dimensional hypercube $Q_d$ has
vertex set $\{0,1\}^d$, with two vertices adjacent iff they differ in exactly one coordinate.
Thus $Q_d$ is $d$-regular and bipartite with (unique) bipartition $\cE \cup \cO$, where $\cE$ and
$\cO$ are the sets of even and odd vertices (the \emph{parity} of $x$ being the parity of
the number of $1$'s in $x$).

For a graph $G$ and disjoint $X,Y\sub V(G)$,
$\nabla(X,Y)$ is the set of edges joining $X $ and $Y$.
We also use $N(X)$ for the set of neighbors of $X$ (with $N(\{x\})=N_x$), $N^2(X)$ for $N(N(X))$
and $B(X)$ for $\{y:N(y)\sub X\}$.
(We will only use these when $G$ is bipartite with $X$ contained in one side of
the bipartition, so e.g.\ will not need to worry about whether $N(X)$
can include vertices of $X$.)

For a graph $G$ and positive integer $k$, say
$u,v \in V=V(G)$ are \textit{$k$-linked} if there is a path from $u$ to $v$ of length at most $k$,
and $A \subset V$ is \textit{k-linked} if for any $u, v \in A$, there are
vertices $u=u_0, u_1, \ldots, u_l=v$ in $A$ such that $u_{i-1}, u_i$ are $k$-linked for each $i\in [l]$.
Then for $A\sub V$ the \textit{$k$-components} of $A$ are its maximal $k$-linked subsets.
(So we use ``component" for a set of vertices rather than a subgraph.)
In what follows we will only be interested in $k=2$, and use $i_A$ for the number
of $2$-components of $A$.  Notice that
\[   
\mbox{distinct 2-components of $A$ have disjoint neighborhoods.}
\]

In what follows $f$ will always be a (proper) coloring of $Q_d$.
We use $f_u$ for the value of $f$ at $u$ and $f_U$ for the restriction
of $f$ to $U$.

We almost always use \emph{lower case letters for the cardinalities of the sets
denoted by the corresponding upper case letters} (thus $a=|A|$, $\hat a =\hat A$ and so on),
usually without comment.

We use $\log$ for $\log_2$.
Following a common abuse, we pretend all large numbers are integers, to avoid
cluttering the paper with irrelevant floor and ceiling symbols.

\subsection{Main point}\label{MP}

For $f \in \cF^*$, we denote by $A_f$ and $\hat{A}_f$ the unions of (resp.) the nonsingleton
and singleton 2-components of $X_f\cap \eee$, and set $G_f=N(A_f)$ and $\hat G_f=N(\hat{A_f})$.
Set
\[ \cF^*(g, \hat g)=\{f \in \cF^*: |G_{f}|=g, |\hat G_{f}|=\hat g\}.\]
The next lemma is almost all of the story.

\begin{lemma}\label{mp}
\[
2^{-N}|\cF^*(g,\hat g)|=
\begin{cases}
2^{-\gO(d)} & \mbox{if $g=0$ and $\hat g \leq d^2/ \log d$,}\\
\exp[-\gO(\tfrac{g}{\log d} + \tfrac{\hat g}{d}\log\tfrac{\hat g}{d})]&\mbox{otherwise.}
\end{cases}
\]
\end{lemma}
\nin

\mn

We close this section with the derivation of \eqref{|f*|} from Lemma~\ref{mp}.
The lemma itself is proved in Section~\ref{Proof}, following the
review of preliminaries in Section~\ref{sec:tools}.

\begin{proof}[Proof of \eqref{|f*|}]
We show
\beq{calc}
2^{-N}|\f^*| = 2^{-\gO(d/\log d)},
\enq
which a little more care with the bounds in Lemma~\ref{mp} (see the remark
following ``$A^l$ terms" in Section~\ref{Combining})
would improve to $2^{-\gO(d)}$.
With
$\sum^*$ running over $(g,\hat g)$ satisfying $g\neq 0$ or $\hat g> d^2/\log d$,
the lemma gives
\[
\mbox{$2^{-N}|\f^*| = 2^{-N}\sum^* |\f^*(g,\hat g)| + (d^2/\log d)2^{-\gO(d)}$;}
\]
so we are just interested in the sum, which we may bound by
\beq{2sums}
\frac{d^2}{\log d}\sum_{g\geq d}2^{-\gz g/\log d}
+\left(1+\sum_{g\geq d}2^{-\gz g/\log d}\right)\left(\sum_{\hat g\geq d^2/\log d}2^{-\gz (\hat g/d)\log(\hat g/d)}\right),
\enq
where $\gz>0$ is the implied constant in the second line of Lemma~\ref{mp}.
(Of course if $g$ is not zero then it is at least $d$.)

With $x=2^{-\gz /\log d}$, the first sum in \eqref{2sums} is
\[
x^d\sum_{i\geq 0}x^i = x^d/(1-x) =x^d\cdot O(\log d) = 2^{-\gO(d/\log d)}.
\]
Similarly, with $y= 2^{-\gz \log(d/\log d)/d}$, the last sum in \eqref{2sums} is less than
\[
\sum_{\hat g\geq d^2/\log d}y^{\hat g} = y^{d^2/\log d}/(1-y)
= 2^{-\gO(d)}.
\]
So we have \eqref{calc}.\end{proof}

\section{Tools}\label{sec:tools}

\subsection{Basic basics}\label{Basics}

Recall
that a \emph{composition} of $m$ is a sequence $(a_1\dots a_s)$ of positive integers
with $\sum a_i=m$ (the $a_i$'s are the \emph{parts} of the composition), and that:

\begin{prop}\label{compprop}
The number of compositions of $m$ is $2^{m-1}$ and the number with at most $b\leq m/2$
parts is $\sum_{i<b}\C{m-1}{i}<\exp_2[b\log (em/b)]$.
\end{prop}

We will use the next lemma in bounding the numbers of certain types of $2$-linked sets in $Q_d$.
It follows from the fact (see e.g. \cite[p.\ 396, Ex.11]{Knuth}) that the
infinite $\Delta$-branching rooted tree contains precisely
\[\frac{{{\Delta n} \choose n}}{(\Delta-1)n+1} \le (e\Delta)^{n-1}\]
rooted subtrees with $n$ vertices.

\begin{lemma}\label{treelemma}
If $G$ is a graph with maximum degree $\Delta$, then the number of $n$-vertex
subsets of $V(G)$ which contain a fixed vertex and induce a connected subgraph is at most $(e\Delta)^n$.
\end{lemma}

\begin{prop}\label{iXprop}
For any $Y\sub \eee$ and $x,b\in \mathbb Z^+$ with $b\leq |Y|/2$, the
number of possibilities for an $X\sub \eee$ with $|X|=x$, $i_X\leq b$
and each 2-component of $X$ meeting $Y$ is at most
\beq{bd-1}
\Cc{|Y|}{b} d^{O(x)}
\enq
(and similarly with $\eee$ replaced by $\Oh$).
\end{prop}

\begin{proof}
The number of possibilities for the (say ordered, though this overcounts) list of sizes, say $x_1\dots x_t$,
of the 2-components of $X$
is at most the number of compositions of $x$, so at most $2^{x-1}$ (see Proposition~\ref{compprop}).
Given this list---so also $i_X$---the number of ways to choose ``roots" in $Y$
for the 2-components is at most $\C{|Y|}{i_X}\leq \C{|Y|}{b}$,
and then Lemma~\ref{treelemma} bounds
the number of ways to complete the 2-components by $(ed^2)^{\sum x_i}=d^{O(x)}$,
which absorbs the initial $2^{x-1}$.\end{proof}

\subsection{Isoperimetry}

As is common in this area, we will need to know a little about
isoperimetric behavior of small subsets of $Q_d$:
\begin{lemma}\label{1overd}
For $A $ a subset of $\cE$ or $\cO$ and $k=d^{o(1)}$,
\[ \mbox{if } a:=|A|=d^k, \mbox{ then } |N(A)|> (1-o(1))(ad/k).\]
\end{lemma}

\begin{proof}
This is similar to \cite[Lemma 6.1]{G}---and a routine application of \cite{KW}---so we will
be brief, referring to \cite{G} for some elaboration.

It is of course enough to consider $A\sub \eee$.
By the main theorem of \cite{KW} (see \cite[Lemma 1.10]{G})
we may assume $A$ is an even Hamming ball; that is,
\beq{balls}
B(v,l)\sub A\sub B(v,l+2),
\enq
for some $v$ and $l$ with $l\equiv |v|\pmod{2}$
(where $|v| =\sum v_i$ and, with $\rho$ denoting distance, $B(v,r)=\{ w \in \cE:\rho(v,w)\leq r\}$
is the even Hamming ball of radius $r$ about $v$).
We just discuss $v\in \eee$, in which case we may assume $v=\underline{0}$.

Elementary calculations show that (assuming $a$ is as in the lemma)
the $l$ in \eqref{balls} is asymptotic to $k$
(since for $l=d^{o(1)}$, $|\C{[d]}{\leq l}\cap \eee|  = d^{l-o(l)}$).
It's then easy to see that each of $|N(B(\underline{0}, l))|,|N(B(\underline{0}, l+2))|$
is asymptotic to $ad/k$, and the lemma follows.\end{proof}

\subsection{Entropy} \label{Entropy}
We next briefly recall relevant entropy background; see e.g.\ \cite{McE} for a less
hurried introduction.

Let $X, Y$ be discrete random variables. The binary entropy of $X$ is
\[ H(X)=\sum_x p(x) \log {\frac{1}{p(x)}},\]
where $p(x)=\pr(X=x)$ (and, recall, $\log$ is $\log_2$).
The \emph{conditional entropy of $X$ given $Y$} is
\beq{condent}
H(X|Y)=\sum_y p(y)\sum_x p(x|y)\log \frac{1}{p(x|y)}
\enq
(where $p(x|y)=\pr(X=x|Y=y)$).

The next lemma lists a few basic properties.

\begin{lemma}\label{entropyprop}~
\begin{enumerate}
\setlength\itemsep{.5em}
\item[{\rm (a)}]
$H(X) \le \log |\mbox{Range}(X)|$, with equality iff $X$ is uniform from its range;

\item[{\rm (b)}]
$H(X,Y)=H(X)+H(Y|X)$;

\item[{\rm (c)}]
$H(X_1\dots X_n|Y) \le \sum H(X_i|Y)$ (note $(X_1\dots X_n)$ is a discrete r.v.);

\item[{\rm (d)}]
if $Z$ is determined by $Y$, then
$H(X|Y) \le H(X|Z)$.
\end{enumerate}
\end{lemma}

We also need the following version of \emph{Shearer's Lemma} \cite{Sh}.
(This statement is more general than the original, but is easily
extracted from the proof in \cite{Sh}.)
\begin{lemma}\label{lem:Sh}
If $X=(X_1, \dots, X_k)$ is a random vector and $\alpha:2^{[k]}\rightarrow \mathbb R^+$
satisfies
\beq{fractiling}
\sum_{A\ni i}\alpha_A =1 \quad \forall i\in [k],
\enq
then
\beq{Shearer}
H(X)\leq \sum_{A\subseteq [k]}\alpha_AH(X_A),
\enq
where $X_A=(X_i:i\in A)$.
\end{lemma}

\subsection{Sapozhenko and Galvin}\label{SandG}

\mn
Finally we recall what we need from the aforementioned results of Sapozhenko and Galvin
(adapted to present purposes; see remarks following Lemma~\ref{agbh}).

\mn

For $A \subset V$, the \textit{closure} of $A$ is $[A]=\{x \in V : N(x) \subseteq N(A)\}$.
Given $A$ (always a subset of $\eee$ or $\Oh$),  we use $G$ and $B$ for $N(A)$ and
$B(A)$ ($:=\{y:N(y)\sub A\}$).

Let
\beq{cGag}
\cG( g)=\{A \subset \cE \mbox{ 2-linked }:  |G|=g\}
\enq
and
\[
\cH(g, b)=\{A \subset \cE \mbox{ 2-linked }: |G|=g, |B|=b \}.
\]

\mn

The first of our lemmas here is from \cite{Sap87} but, as mentioned earlier,
we refer to the more accessible \cite[Lemma 3.1]{GS}:
\begin{lemma}\label{ag}
For any $\gamma <2$ and each $g\in [d^4,\gamma^d]$,
\[|\cG(g)| \le 2^{g-\gO(g/\log d)}.\]
\end{lemma}

\mn

The next lemma was originally a step in the proof of
Lemma~\ref{ag}, but will also play an independent role below.

\begin{lemma}[\cite{GS}, Lemmas 5.3-5.5]\label{approx}
For $g$ as in Lemma~\ref{ag} and $\g=\g(g)$, there are $\W=\cW( g)\sub 2^\cE \times 2^\cO$
with
\[|\cW| = 2^{O(g\log ^2d/d)}\]
and $\vp=\vp_g:\g\ra \W$ such that for each $A\in \g$,
$(S,F):=\vp(A)$ satisfies:
\begin{enumerate}
\setlength\itemsep{.5em}
\item[{\rm (a)}] $S \supseteq [A]$, $F \subseteq G$;
\item[{\rm (b)}]  $d_F(u) \ge d-d/\log d$ $\forall u \in S$.
\end{enumerate}
\end{lemma}

\mn

\begin{lemma}[\cite{G}, Lemma 7.1]\label{agbh}
For any $\gamma <2$ and each $g, b\le \gamma^d$,
\[|\cH(g, b)| < 2^{d}2^{g-b-\gO(g/\log d)}.\]
\end{lemma}

\mn
\emph{Remarks.}
The preceding lemmas are special cases/consequences of
the cited results from \cite{G} and \cite{GS}, with statements somewhat simplified.
In particular we have omitted some parameters for the $\g(\cdot)$'s, e.g.\ $a$ ($:=|A|$)
and $v$ (a fixed vertex which $A$ must contain).
Dropping these specifications just multiplies bounds by a (for us) negligible $2^d$.
Also, \cite[Lemma 3.1]{GS} (the more general version of our Lemma~\ref{ag})
assumes a lower bound on $g-a$, which in our situation is
much less than what follows from Lemma~\ref{1overd}.

\section{Proof}\label{Proof}

\subsection{Orientation}\label{Preview}

We first spend a little time trying to motivate what's happening below,
hoping this makes the discussion easier to follow.
For purposes of comparison we begin with a standardish entropy-based bound.

\mn

Given $\g\sub \f^*$, set
\begin{eqnarray}\label{def_tu}
T(u)=T_\g(u) &=&
\tfrac{1}{d}H(f_{N_u})+H(f_u|f(N_u))\nonumber\\
&=&\tfrac{1}{d}[H(f(N_u))+H(f_{N_u}|f(N_u))]+H(f_u|f(N_u)),~~
\end{eqnarray}
where $f$ is uniform from $\g$.
(We will use this only with $u\in \Oh$.)
Then
\begin{eqnarray*}
\log |\g|=H(f)&=&
H(f_\cE)+H(f_\cO|f_\cE)\\
&\le &\frac{1}{d}\sum_{u\in \Oh} H(f_{N_u})+\sum_{u\in \Oh} H(f_u|f(N_u))
=\sum_{u\in \Oh}T(u).
\end{eqnarray*}
\nin
Here the first two equalities are given by
(a) and (b) of Lemma~\ref{entropyprop} and the inequality by
Lemmas~\ref{lem:Sh} and \ref{entropyprop} (c,d),
the former with
\beq{alphaS}
\ga_S=\left\{\begin{array}{ll}
1/d&\mbox{if $S=N_u$ for some $u\in \Oh$,}\\
0&\mbox{otherwise.}
\end{array}\right.
\enq

On the other hand, for each possible value $c$ of $f(N_u)$,
\begin{eqnarray*}
H(f_u|f(N_u)=c) &\leq & \log(4-|c|),\\
H(f_{N_u}|f(N_u)=c) &\leq & d\log|c|.
\end{eqnarray*}
Since $\log x+\log (4-x)\leq 2$,
this bounds
$\frac{1}{d}H(f_{N_u}|f(N_u))+H(f_u|f(N_u))$
(the main part of \eqref{def_tu}) by
\[
\mbox{$\sum_c\pr(f(N_u)=c)[\tfrac{1}{d}H(f_{N_u}|f(N_u)=c)+H(f_u|f(N_u)=c)]\leq 2$,}
\]
yielding
\beq{Tu}
T(u) \leq 2 +O(1/d)
\enq
(since $H(f(N_u))=O(1)$).

In particular, applying this with $\g=\f^*$ gives the easy bound
\beq{logf*}
\log|\f^*| \leq N +O(N/d),
\enq
whereas we want $\log|\f^*| < N-\go(1)$;
so what we do below may be thought of as fighting over this difference.
(Note this argument makes no use of
the fact that members of $\f^*$ are non-ideal, so can't give a bound less than $N$.)

\mn

We now very briefly sketch the actual argument.
We think of $|\f^*|$ as the number of ways to specify $f\in\f^*$,
which we do in two stages.
The first of these identifies a ``template," $\T=\T_f$, which provides some, usually
incomplete, information on $X_f$ (recall this is the set of vertices that are bad for $f$).
In fact $\T$ will completely specify $X_f\cap \eee$, but the information on $X_f\cap \Oh$
will typically be less precise.

The second (``coloring") stage then treats possibilities for $f$
given $\T$.  Thus we restrict to a set $\g$ of $f$'s satisfying $\T_f=\T$,
usually with some ``cheap" part of $f$ also specified,
and return to the entropy approach leading to \eqref{logf*}.
The hope---and basic idea of the proof---is that what we save
in the above argument by exploiting information from the template
recovers (more exactly, \emph{more than} recovers) what we've paid for said information.

\mn

In what follows we usually speak in terms of the \emph{cost} of a choice,
meaning the log of the number of possibilities for that choice, which we think of as the number of
bits ``paid" for the desired information.

\subsection{Templates}\label{Templates}

A template will consist of two parts, the first specifying
$X_f\cap \eee$ and the second corresponding to,
but not necessarily precisely identifying,
the portion of $X_f\cap \Oh$ not adjacent to $X_f\cap \eee$.
(For perspective we note that the asymmetry between $\eee$ and $\Oh$
corresponds to \eqref{f*def} in the definition of $\f^*$, an assumption
we will use frequently below.)

\mn

Names for the sets involved will now be helpful;
for a particular $f$ we use the following notation,
with dependence on $f$ suppressed (so $X=X_f$, $A_i=A_i(f)$ and so on).

\mn

$A_i$'s: non-singleton 2-components of $X\cap \cE$;

$\hat A_i$'s: singleton 2-components of $X \cap \cE$;

$G_i=N(A_i)$, $\hat G_i=N(\hat A_i)$;

$A=\cup A_i$ and similarly for $\hat A$, $G$ and $\hat G$ (as in the passage preceding
Lemma~\ref{mp});

$\R=\Oh\sm (G \cup \hat G)$;

\mn

$P_i$'s: 2-components of $X \cap \R$ meeting $N^2(G \cup \hat G)$;

$\bar P_i$'s:
non-singleton 2-components of $X \cap\R$ not meeting $N^2(G \cup \hat G)$;

$\hat P_i$'s: singleton 2-components of $X \cap \R$ not in $N^2(G \cup \hat G)$;

$Q_i=N(P_i)$, $\bar Q_i=N(\bar P_i)$ and $\hat Q_i=N(\hat P_i)$;

$P=\cup P_i$, $Q=\cup Q_i$ etc.

\mn
(See figure \ref{fig1}.) Note that the vertices of $Q\cup \bar Q\cup \hat Q$, not being in $A \cup \hat A$,
are all good, while the template does not usually distinguish good and bad vertices of $G\cup \hat G$.
(The one exception to this is in the treatment of the special case \eqref{except}
in Section~\ref{Finally}.)
Note also that the
$G_i$'s and $\hat G_i$'s are pairwise disjoint and similarly for the $Q_i$'s, $\bar Q_i$'s and $\hat Q_i$'s.

\begin{figure}
\begin{center}

\begin{tikzpicture}[scale=0.9, main node/.style={circle,fill=blue!20,draw,minimum size=0cm,inner sep=0pt]}]
    \node[main node] (1) at (-0.5,0) {};
    \node[main node] (2) at (0.5,0)  {};
    \node[main node] (3) at (2.5,0) {};
    \node[main node] (4) at (3,0) {};
    \node[main node] (5) at (5,0) {};
    \node[main node] (6) at (5.5,0) {};
    \node[main node] (7) at (7.5,0) {};
    \node[main node] (8) at (8,0) {};
    \node[main node] (9) at (10,0)  {};
    \node[main node] (10) at (11.5,0) {};
    \node[main node] (11) at (13.5,0) {};
    \node[main node] (12) at (14.5,0) [label=right:$\quad \mathcal E$] {};

    \node[main node] (42) at (0.5,-.4)  {};
    \node[main node] (43) at (2.5,-.4) {};
    \node[main node] (44) at (3,-.4) {};
    \node[main node] (45) at (5,-.4) {};
    \node[main node] (46) at (5.5,-.4) {};
    \node[main node] (47) at (7.5,-.4) {};
    \node[main node] (48) at (8,-.4) {};
    \node[main node] (49) at (10,-.4)  {};
    \node[main node] (50) at (11.5,-.4) {};
    \node[main node] (51) at (13.5,-.4) {};

    \node at (1.5,.2) [label=below:$_{\bar Q}$] {};
    \node at (4,.2) [label=below:$_{\hat Q}$] {};
    \node at (6.5,.2) [label=below:$_{ Q}$] {};
    \node at (9,.2) [label=below:$_{A}$] {};
    \node at (12.5,.2) [label=below:$_{\hat A}$] {};

    \node[main node] (21) at (-0.5,1.5)  {};
    \node[main node] (22) at (1,1.5)  {};
    \node[main node] (23) at (2,1.5)  {};
    \node[main node] (24) at (3.5,1.5)  {};
    \node[main node] (25) at (4.5,1.5) {};
    \node[main node] (26) at (6,1.5) {};
    \node[main node] (27) at (7,1.5) {};
    \node[main node] (28) at (7.5,1.5) {};
    \node[main node] (29) at (10.5,1.5) {};
    \node[main node] (30) at (11, 1.5) {};
    \node[main node] (31) at (14, 1.5) {};
    \node[main node] (32) at (14.5, 1.5) [label=right:$\quad \mathcal O$] {};

    \node[main node] (62) at (1,1.9)  {};
    \node[main node] (63) at (2,1.9)  {};
    \node[main node] (64) at (3.5,1.9)  {};
    \node[main node] (65) at (4.5,1.9) {};
    \node[main node] (66) at (6,1.9) {};
    \node[main node] (67) at (7,1.9) {};
    \node[main node] (68) at (7.5,1.9) {};
    \node[main node] (69) at (10.5,1.9) {};
    \node[main node] (70) at (11, 1.9) {};
    \node[main node] (71) at (14, 1.9) {};

    \node at (1.5,2.1) [label=below:$_{\bar P}$] {};
    \node at (4,2.1) [label=below:$_{\hat P}$] {};
    \node at (6.5,2.1) [label=below:$_{ P}$] {};
    \node at (9,2.1) [label=below:$_{G}$] {};
    \node at (12.5,2.1) [label=below:$_{\hat G}$] {};

    \node[circle,fill=black,draw,minimum size=0.1cm,inner sep=0pt] (100) at (8.5,1.5) {};
    \node[circle,fill=black,draw,minimum size=0.1cm,inner sep=0pt] (200) at (6,0) {};
    \node[circle,fill=black,draw,minimum size=0.1cm,inner sep=0pt] (101) at (12,1.5) {};
    \node[circle,fill=black,draw,minimum size=0.1cm,inner sep=0pt] (201) at (7,0) {};

    \draw[dashed] (100) -- (200) node [below] {};
    \draw[dashed] (101) -- (201) node [below] {};

    \draw (1) -- (2) node [below] {};
    \draw (2) -- (3) node [below] {};
    \draw (3) -- (4) node [below] {};
    \draw (4) -- (5) node [below] {};
    \draw (5) -- (6) node [below] {};
    \draw (6) -- (7) node [below] {};
    \draw (7) -- (8) node [below] {};
    \draw (8) -- (9) node [below] {};
    \draw (9) -- (10) node [below] {};
    \draw (10) -- (11) node [below] {};
    \draw (11) -- (12) node [below] {};
    \draw (21) -- (22) node [below] {};
    \draw (22) -- (23) node [below] {};
    \draw (23) -- (24) node [below] {};
    \draw (24) -- (25) node [below] {};
    \draw (25) -- (26) node [below] {};
    \draw (26) -- (27) node [below] {};
    \draw (27) -- (28) node [below] {};
    \draw (28) -- (29) node [below] {};
    \draw (29) -- (30) node [below] {};
    \draw (30) -- (31) node [below] {};
    \draw (31) -- (32) node [below] {};

    \draw (2) -- (22) node [below] {};
    \draw (3) -- (23) node [below] {};
    \draw (4) -- (24) node [below] {};
    \draw (5) -- (25) node [below] {};
    \draw (6) -- (26) node [below] {};
    \draw (7) -- (27) node [below] {};
    \draw (8) -- (28) node [below] {};
    \draw (9) -- (29) node [below] {};
    \draw (10) -- (30) node [below] {};
    \draw (11) -- (31) node [below] {};
    \draw (2) -- (1) node [below] {};
    \draw (2) -- (1) node [below] {};
    \draw (2) -- (1) node [below] {};
    \draw (2) -- (1) node [below] {};
    \draw (2) -- (1) node [below] {};
    \draw (2) -- (1) node [below] {};
    \draw (2) -- (1) node [below] {};
    \draw (2) -- (1) node [below] {};
    \draw (2) -- (1) node [below] {};
    \draw (2) -- (1) node [below] {};
    \draw (2) -- (1) node [below] {};

    \draw (2) -- (42) node [below] {};
    \draw (3) -- (43) node [below] {};
    \draw (4) -- (44) node [below] {};
    \draw (5) -- (45) node [below] {};
    \draw (6) -- (46) node [below] {};
    \draw (7) -- (47) node [below] {};
    \draw (8) -- (48) node [below] {};
    \draw (9) -- (49) node [below] {};
    \draw (10) -- (50) node [below] {};
    \draw (11) -- (51) node [below] {};

    \draw (22) -- (62) node [below] {};
    \draw (23) -- (63) node [below] {};
    \draw (24) -- (64) node [below] {};
    \draw (25) -- (65) node [below] {};
    \draw (26) -- (66) node [below] {};
    \draw (27) -- (67) node [below] {};
    \draw (28) -- (68) node [below] {};
    \draw (29) -- (69) node [below] {};
    \draw (30) -- (70) node [below] {};
    \draw (31) -- (71) node [below] {};

    \draw (42) -- (43) node [below] {};
    \draw (44) -- (45) node [below] {};
    \draw (46) -- (47) node [below] {};
    \draw (48) -- (49) node [below] {};
    \draw (50) -- (51) node [below] {};

    \draw (62) -- (63) node [below] {};
    \draw (64) -- (65) node [below] {};
    \draw (66) -- (67) node [below] {};
    \draw (68) -- (69) node [below] {};
    \draw (70) -- (71) node [below] {};   	    		    	
		
\end{tikzpicture}

\end{center}
\caption{} \label{fig1}
\end{figure}

\mn

Treatment of the contributions (to our overall cost) of the above pieces
will depend on their sizes, necessitating some further decomposition, as follows.
(Recall $a_i=|A_i|$ and so on.)  Say
\[
\mbox{$A_i$ is}\left\{\begin{array}{ll}
\mbox{\emph{small} $~$ if $g_i < \exp_2[\log^3d]$ and }\\
\mbox{\emph{large} $~~$ otherwise,}
\end{array}\right.
\]
and similarly for $\bar P_i$ and $P_i$.
Let $A^s,A^l$ be the unions of the small and large $A_i$'s (resp.)
and extend this notation in the natural ways; thus $G^s=N(A^s)$, $P^l$ is the
union of the large $P_i$'s and so on.
We also set $b=|B(A^l)|$ (see Section~\ref{SandG}).

\mn
\emph{Remark.}
The choice $\exp_2[\log^3d]$ is not delicate.  The most serious constraint
is in the discussion of \eqref{Plbd}, where we use $q^l =d^{\go(\log d)}$.
The other cutoffs could be smaller---we mainly need them to support application of
Lemmas~\ref{ag}-\ref{agbh}---but for simplicity we use one value for all.

\mn

Note that in proving Lemma~\ref{mp} we are given $g $ and $\hat g$.
Analysis in Sections~\ref{Combining} and \ref{Finally} will vary depending on these,
but for now the discussion is general.
It will be convenient to set $\bg=g+\hat g$.

Before proceeding, we set aside the easy (but important) special case in which
\beq{except}
\mbox{$a=\bar p=p=0~ $ and $~ \hat g\leq d^2/\log d.$}
\enq
This will be handled in Section~\ref{Finally},
and \emph{until then we restrict to $f$'s that are not of this type.}

\mn

It will be helpful to have specified the sizes of some of the other sets above,
namely
\[
\mbox{$a^s ~ (=|A^s|)$, $g^s$,
$p$, $q$, $\bar p$, $\bar q$, $\hat p$, $i_{A^s}$, $i_{\bar P^s}$, $i_{P^s}, b$,}
\]
\mn
which we may do at an (eventually negligible) cost of
\beq{params}
O(\log \bg).
\enq
(Most of these could be skipped, but it's easier to pay the above negligible cost
up front than to waste time on this issue.)

\mn

We begin with costs associated with the non-large sets above.
These choices are mostly treated as if made autonomously;
that is, without trying to exploit proximity or non-proximity of different pieces.
The one exception is in the cost of $P^s$,
where we sometimes
save substantially by
choosing initial vertices for the 2-components from $N^2(G\cup \hat G)$ rather than all of $\Oh$.

\begin{claim}\label{non-large}
The costs of identifying $\hat A$,$A^s$,$\hat P$,$\bar P^s$ and $P^s$ are bounded by:

\begin{itemize}
\setlength\itemsep{.5em}
\item[{\rm [$\hat A$]}]  $~~\log\Cc{N/2}{\hat a} \leq \hat g -\gO((\hat g/d))\log (\hat g/d))$
(using $\hat g =\hat a d$);

\item[{\rm [$A^s$]}]  $~~i_{A^s}(d-1) + O(a^s \log d)$;

\item[{\rm [$\hat P$]}]  $~~\log\Cc{N/2}{\hat p} \leq \hat p d = \hat q$;

\item[{\rm [$\bar P^s$]}]  $~~i_{\bar P^s}(d-1) + O(\bar p^s \log d$);

\item[{\rm [$P^s$]}]  $~~i_{P^s}\log (e\bg d^2/i_{P^s}) +O(p^s \log d) $.

\end{itemize}
\end{claim}

\begin{proof}
The first and third of these are trivial and the others are instances of
Proposition~\ref{iXprop}, with some relaxation of bounds.
We use $Y=\Oh$ ($|Y|=N/2$) for all but $[P^s]$, where,
as mentioned above, we save significantly by
taking $Y=N^2(G\cup \hat G)$.\end{proof}

For larger pieces we have the following bounds, which, in contrast to the
elementary Claim~\ref{non-large}, depend on the sophisticated results
of Section~\ref{SandG}.

Lemmas~\ref{agbh} and \ref{ag}
bound the costs of $A^l$ and $\bar Q^l$ by
\beq{Alcost}
g^l - b-\gO(g^l/\log d)
\enq
and
\beq{barQlcost}
\bar q^l - \gO(\bar q^l/\log d).
\enq
(E.g.\ for \eqref{Alcost}:  we first pay $O(g^l\log d/d)$ for the list of $g_i$'s
and (with the obvious meaning) $b_i$'s
corresponding to large $A_i$'s (the cost bound given by Proposition~\ref{compprop}, using
$i_{A^l}< g^l/d$), and then apply Lemma~\ref{agbh} to the pieces,
absorbing the initial $O(g^l\log d/d)$ and the $2^d$ from the lemma
in the ``$\gO$" term of \eqref{Alcost}.)

\mn

For $P^l$, perhaps the most interesting part of this story,
the cost of full specification turns out to be more than we can afford,
and we retreat to the approximations of
Lemma~\ref{approx}.
(As mentioned earlier, Lemma~\ref{approx} was originally
a \emph{step} in the proof of Lemma~\ref{ag}; so its present appearance in a non-auxiliary role
seems interesting.)

Here again we pay an initial
\beq{pLis}
O(q^l\log d/d)
\enq
for $(q_i:i\in I)$, where $I$ indexes the large $P_i$'s.
Then for $i\in I$ we slightly modify the output of Lemma~\ref{approx}
(applied here with the roles of $\eee$ and $\Oh$ reversed),
letting $(S_i',F_i) =\vp_i(P_i)\in \W_i$,
with $\vp_i=\vp_{q_i}$ and $\W_i=\W(q_i)$ as in the lemma,
and setting $S_i=S_i'\sm (G\cup G')$ (see Figure 2).
Note $(S_i,F_i)$ still enjoys the properties
the lemma promised for $(S_i',F_i)$; that is,
\beq{Pi1}
S_i \supseteq [P_i],  ~~~F_i \subseteq Q_i,
\enq
\beq{Pi2}
d_{F_i}(u) \ge d-d/\log d ~~~ ~\forall u \in S_i.
\enq

\mn
(The only thing to observe here---used for the first part of \eqref{Pi1}---is
that $[P_i]\cap (G\cup \hat G)=\0$
follows from $N(P_i)\cap (A\cup \hat A)=\0$.
Incidentally, $S\supseteq P_i$ in \eqref{Pi1} would be enough for our purposes.)
Note also that the $S_i$'s are pairwise disjoint (by
\eqref{Pi2} since the second part of \eqref{Pi1} implies the $F_i$'s are pairwise disjoint)
and that, with $S=\cup S_i$, $F=\cup F_i$,
\beq{SandF}
S\cap (G\cup \hat G)= \0 = F\cap (\bar Q\cup \hat Q\cup Q^s).
\enq

\begin{figure}
\begin{center}

\begin{tikzpicture}[scale=0.9, main node/.style={circle,fill=blue!20,draw,minimum size=0cm,inner sep=0pt]}]

    \node[main node] (1) at (-0.5,0) {};
    \node[main node] (2) at (0.5,0)  {};
    \node[main node] (3) at (2.5,0) {};
    \node[main node] (4) at (3.5,0) {};    
    \node[main node] (5) at (5.5,0) {};
    \node[main node] (6) at (7,0) {};

    \node[main node] (12) at (6.7,0) [label=right:$\quad \mathcal E$] {};

    \node[main node] (42) at (0.5,-.4)  {};
    \node[main node] (43) at (2.5,-.4) {};
    \node[main node] (44) at (3.5,-.4) {};    
    \node[main node] (45) at (5.5,-.4) {};

    
    \node at (1.5,-.2) [label=below:$_{Q^l}$] {};
    \node at (4.5,-.2) [label=below:$_{A \cup \hat A}$] {};

    \node[main node] (21) at (-0.5,1.5)  {};
    \node[main node] (22) at (1,1.5)  {};
    \node[main node] (23) at (2,1.5)  {};
    \node[main node] (24) at (3,1.5)  {};    
    \node[main node] (25) at (6,1.5) {};
    \node[main node] (26) at (7,1.5) {};

    \node[main node] (32) at (6.7, 1.5) [label=right:$\quad \mathcal O$] {};


    \node[main node] (62) at (1,1.9)  {};
    \node[main node] (63) at (2,1.9)  {};
    \node[main node] (64) at (3,1.9)  {};    
    \node[main node] (65) at (6,1.9) {};

    \node at (1.5,2.1) [label=below:$_{P^l}$] {};
    \node at (4.5,2.1) [label=below:$_{G \cup \hat G}$] {};

    \node[circle,fill=black,draw,minimum size=0.1cm,inner sep=0pt] (100) at (3.8,1.5) {};
    \node[circle,fill=black,draw,minimum size=0.1cm,inner sep=0pt] (200) at (1.8,0) {};
    
    \draw[dashed] (100) -- (200) node [below] {};

    \draw (1) -- (2) node [below] {};
    \draw (2) -- (3) node [below] {};
    \draw (3) -- (4) node [below] {};
    \draw (4) -- (5) node [below] {};
    \draw (5) -- (6) node [below] {};

    \draw (21) -- (22) node [below] {};
    \draw (22) -- (23) node [below] {};
    \draw (23) -- (24) node [below] {};
    \draw (24) -- (25) node [below] {};
    \draw (25) -- (26) node [below] {};

    \draw (2) -- (22) node [below] {};
    \draw (3) -- (23) node [below] {};
    \draw (4) -- (24) node [below] {};
    \draw (5) -- (25) node [below] {};

    \draw (2) -- (42) node [below] {};
    \draw (3) -- (43) node [below] {};
    \draw (4) -- (44) node [below] {};
    \draw (5) -- (45) node [below] {};

    \draw (22) -- (62) node [below] {};
    \draw (23) -- (63) node [below] {};
    \draw (24) -- (64) node [below] {};
    \draw (25) -- (65) node [below] {};

    \draw (42) -- (43) node [below] {};
    \draw (44) -- (45) node [below] {};

    \draw (62) -- (63) node [below] {};
    \draw (64) -- (65) node [below] {};


	\node[main node] (snw) at (0.8,2.1) {};
	\node[main node] (sne) at (2.2,2.1) {};
	\node[main node] (ssw) at (0.8,1.3) {};
	\node[main node] (sse) at (2.2,1.3) {};
	
	\node[main node] (fsw) at (0.8,-0.2) {};
	\node[main node] (fse) at (2.3,-0.2) {};
	\node[main node] (fnw) at (0.8,.2) {};
	\node[main node] (fne) at (2.3,.2) {};
	
	\draw[dashed] (snw) -- (ssw) node [below] {};
	\draw[dashed] (ssw) -- (sse) node [below] {};
	\draw[dashed] (sse) -- (sne) node [below] {};
	\draw[dashed] (snw) -- (sne) node [below] {};
	
	\draw[dashed] (fnw) -- (fsw) node [below] {};
	\draw[dashed] (fnw) -- (fne) node [below] {};
	\draw[dashed] (fse) -- (fne) node [below] {};
	\draw[dashed] (fse) -- (fsw) node [below] {};
	
	\node at (1.5,1.5) [label=below:$_S$] {};
	\node at (1.5,0) [label=above:$_F$] {};


\end{tikzpicture}
\end{center}
\caption{$S$ and $F$} \label{fig2}
\end{figure}

Lemma~\ref{approx} bounds the cost of specifying the $(S_i,F_i)$'s by
\beq{SFcost}
O(q^l\log^2d/d),
\enq
which absorbs the decomposition cost \eqref{pLis}.

\bn

This completes the template stage (apart from the treatment of \eqref{except}).
Formally---but we won't actually use this---we could say that
$\T=\T_f$
is $(A,\hat A,\bar P,\hat P,P^s, S,F)$.
(Note $A$ determines $A^s, A^l$ and similarly for $\bar P$.)

\subsection{Colors}\label{Colors}

First notice that each
$\bar P_i$, $\bar Q_i$, $\hat P_i$, $\hat Q_i$, $P^s_i$, $Q_i^s$ (these being, of course, the 2-components of $P^s$ and $Q^s$)
and $F_i$ is monochromatic.
(The general observation is:
if $Z$ is a 2-linked subset of $\Oh$ or $\eee$,
all vertices of $Z$ are bad and all vertices of $N(Z)$ are good,
then each of $Z$, $N(Z)$ is monochromatic (and the color for $Z$ determines the color for $N(Z)$).)

So we begin by paying
\beq{ppp}
i_P +i_{\bar P} +\hat p
\enq
to specify the colors of these sets.
(These are the ``cheap" color choices mentioned earlier.)
\emph{We then restrict our discussion to the set $\g$ of
$f$'s agreeing with these specifications} (and the specified $\T$).

\mn

For appraising the cost of identifying a member of $\g$,
we refine the discussion leading to \eqref{logf*}.
To begin, we will in each instance consider $T(u)$ (defined in \eqref{def_tu})
only for the $u$'s in
some subset, say $\U$, of $\Oh$, with the rest of $\ga$ (as in Lemma~\ref{lem:Sh};
\emph{cf.}\ \eqref{alphaS})
supported on singletons.
(We use $u$ and $v$ for vertices of $\Oh$ and $\eee$ respectively.)
Thus we use
\beq{H(f)}
H(f) ~\leq ~\sum_{u\in \U}T(u) +\sum_{u\in \Oh\sm\U}H(f_u|f(N_u))
+\sum_{v\in \eee}(1-d_\U(v)/d)H(f_v).
\enq
\nin
As noted earlier, $\T$ includes specification of $X_f\cap \eee$,
so we know which vertices of $\eee$ are bad for $f$.
A key ingredient in evaluating the first term in
\eqref{H(f)}
is then the following variant of \eqref{Tu},
in which---just to point out that this doesn't require uniform
distribution---$T_\mu(u)$ is the natural generalization of $T_\g(u)$
to the probability distribution $\mu$.
\begin{prop}\label{prop:tu}
If $X\cup Y$ is a partition of $N_u$ with $X,Y\neq \0$, and $f$ is chosen
from some probability distribution $\mu$ on the set of colorings for which $X$ is
entirely good and $Y$ entirely bad, then
\beq{Tuagain}
T_\mu(u) \leq 1+O(1/d).
\enq
\end{prop}

\begin{proof}

This is similar to the derivation of \eqref{Tu}.
Notice that $|f(N_u)|$ must be either 2 or 3
(it is at least 2 by our assumption on $X,Y$ and at most 3 since $f(N_u)\not\ni f_u$),
and that
\[
\mbox{$f(N_u)$ determines} \left\{\begin{array}{ll}
f_{N_u}
&\mbox{if $|f(N_u)| =2$,}\\
f_u &\mbox{if $|f(N_u)| =3$,}
\end{array}\right.
\]
so that
$H(f_{N_u} | f(N_u) = c) = 0 $ if $|c|=2$ and
$H(f_u | f(N_u) = c) = 0 $ if $|c|=3$.
Moreover,
\begin{eqnarray*}
H(f_u | f(N_u) = c) \leq 1 & \mbox{if $|c|=2$,}\\
H(f_{N_u} | f(N_u) = c) \leq d & \mbox{if $|c|=3$}
\end{eqnarray*}
(the $d$ could be replaced by $\max\{|X|,|Y|\}$).

Thus
$
\tfrac{1}{d}H(f_{N_u}|f(N_u))+H(f_u|f(N_u))
$
(the main part of \eqref{def_tu}) is
\[
\mbox{$\sum_c\pr(f(N_u)=c)[\tfrac{1}{d}H(f_{N_u}|f(N_u)=c)+H(f_u|f(N_u)=c)]\leq 1$,}
\]
and Proposition~\ref{prop:tu} follows since
$H(f(N_u))= O(1)$.
\end{proof}
\mn

Of course knowing $X_f\cap \eee$ also bounds the last sum in \eqref{H(f)}
by
\[
\sum_{v\in \eee}(1-d_\U(v)/d) =N/2-|\U|,
\]
which in cases where $\g$ specifies some of the $f_v$'s,
say those in $\V\sub \eee$, improves to
\beq{N2U}
N/2-|\U| -\sum_{v\in \V}(1-d_\U(v)/d)
= N/2-|\U| -|\nabla(\V,\Oh\sm \U)|/d.
\enq

So we will be evaluating \eqref{H(f)} using
\eqref{Tuagain} and \eqref{N2U} (with a small assist from \eqref{Tu}).
From this point we take
\[
\mbox{$\U=G\cup \hat G~$ and $~\V=\bar Q\cup \hat Q\cup Q^s\cup F$}
\]
(so also $\Oh\sm \U=\R$; recall colors for $\V$ were specified at \eqref{ppp}).
We then have the following bounds for the three sums in \eqref{H(f)}.

\mn

The combination of \eqref{Tu} and \eqref{Tuagain} bounds the first by
\beq{C1}
g +b + \hat g +O(\bg/d)
\enq
(using \eqref{Tu} for $u\in B$ and \eqref{Tuagain} for the rest).

\mn

We next claim that the second is at most
\beq{C2}
s\log 3 + N/2-(g+\hat g+s) = N/2 - (g+\hat g ) + s (\log 3 -1)
\enq
(where $s=|S|$).
Here we use $S\cap (G\cup \hat G)=\0$ (see \eqref{SandF}) and
\[
H(f_u|f(N_u))\leq H(f_u) \leq
\left\{\begin{array}{ll}
1&\mbox{if $u\in \Oh\sm (G\cup \hat G\cup S)$,}\\
\log 3&\mbox{if $u\in S$.}
\end{array}\right.
\]
The second bound is trivial.  For the first notice that we actually \emph{know}
$f_u$ if $u\in \bar P\cup \hat P\cup P^s$
and in other cases know $u$ is good (using $P^l\sub S$).

\mn

Finally,
the last term in \eqref{H(f)} is at most
$N/2-|\U| -|\nabla(\V,\R)|/d $
(see \eqref{N2U}), which we rewrite as
\beq{C3}
N/2 -(g+\hat g) - [\bar q+\hat q+|\nabla(Q^s,\R)|/d+|\nabla(F,\R)|/d]
\enq
(using $N(\bar Q\cup \hat Q)\sub \R$ and
$F\cap (\bar Q\cup \hat Q\cup Q^s)=\0$ (see \eqref{SandF})).

\subsection{In sum}\label{Combining}

It remains to check that the above cost bounds give Lemma~\ref{mp}
(in cases not covered by \eqref{except}).
We are now playing the game mentioned near the end of Section~\ref{Preview},
in which we try to balance costs from the template stage against what we have gained
(relative to \eqref{logf*}) in the coloring stage (and need to come out slightly ahead).

\mn

The bounds are:  from the template stage, \eqref{params}
and the more serious bounds in
Claim~\ref{non-large}, \eqref{Alcost}, \eqref{barQlcost} and \eqref{SFcost};
and from the coloring stage, the minor \eqref{ppp}
and the non-minor \eqref{C1}-\eqref{C3}.
We will recall the template bounds as we come to them.
The total cost from the coloring stage is bounded by
\beq{colorsummary}
N+ b - (g+\hat g) + s(\log 3-1) -(\bar q+\hat q) -|\nabla(Q^s\cup F,\R)|/d +O(\bg/d),
\enq
gotten by summing \eqref{C1}-\eqref{C3}
and absorbing \eqref{ppp} in the $O(\bg/d)$.

\bn

Note that both the $O(\bg/d)$ in \eqref{colorsummary}
and the $O(\log \bg)$ in \eqref{params} are
negligible relative to the bounds in Lemma~\ref{mp}.
(The comparison is least drastic when $g=0$ and $\hat g$ is not much more than $d^2/\log d$.)
So we may safely ignore these terms
and in particular, rearranging and slightly expanding, replace \eqref{colorsummary} by
\beq{CS}
N-\hat g-\hat q -g^s -g^l +b-\bar q^s -\bar q^l+s(\log 3-1)-|\nabla(Q^s\cup F,\R)|/d.
\enq
The initial $N$ will of course cancel the $2^{-N}$ in Lemma~\ref{mp},
and we want to show that the combination of the remaining terms in \eqref{CS}
and the template costs produces the savings the lemma promises.
We consider terms in groups of two or three corresponding to the different
constituents of the template, following the order in \eqref{CS},
with the expressions in curly brackets below
representing template costs and those immediately following them
taken from \eqref{CS} (and the right hand sides the bounds we will use).
We first collect all these bounds and then take stock.

\mn
\underline{$\hat A$ terms}:
$~~~
\{\hat g -\gO((\hat g/d))\log (\hat g/d))\} -\hat g =-\gO((\hat g/d))\log (\hat g/d))$

\mn
\underline{$\hat P$ terms}:
$~~~
\{\hat q\} - \hat q =0$

\mn
\underline{$A^s$ terms}:
$~~~\{i_{A^s}(d-1) + O(a^s \log d)\} - g^s \leq -(1/2-o(1))g^s$

\mn
(since $g^s\ge \max \{2i_{A^s}(d-1),\gO(a^sd/\log^2d)\}$, the second bound by Lemma~\ref{1overd})

\mn
\underline{$A^l$ terms}:
$~~~
\{g^l - b-\gO(g^l/\log d)\} -g^l+b = -\gO(g^l/\log d)$

\mn
\emph{Remark.}
Using the last two bounds, we could replace the second bound in
Lemma~\ref{mp} by
$\exp[-\gO(g^s + g^l/\log d + (\hat g/d)\log(\hat g/d))]$
and the bound in \eqref{calc} by $2^{-\gO(d)}$.

\mn
\underline{$\bar P^s$ terms}:
$~~~
\{i_{\bar P^s}(d-1) + O(\bar p^s \log d)\}-\bar q^s \leq -(1/2-o(1))\bar q^s$

\mn
(as for the $A^s$ terms).

\mn
\underline{$\bar P^l$ terms}:
$~~~
\{\bar q^l - \gO(\bar q^l/\log d)\}-\bar q^l \leq 0$

\mn

The $P$ terms require a little more care.  Here we will sometimes incur a small
loss---that is, a positive contribution---but can live with this provided these
losses are negligible relative to
\beq{sillybg}
\bg\min\{d^{-1}\log (\bg/d),(\log d)^{-1}\}
\enq
since our current gain from $\hat A$, $A^s$ and $A^l$ is at least of this order.
Recall from \eqref{C3} that the last term in
\eqref{CS} is the same as
$[|\nabla(Q^s,\R)|+ |\nabla(F,\R)|]/d$.

\mn
\underline{$P^s$ terms}:
\beq{Psterms}
\{i_{P^s}\log (e\bg d^2/i_{P^s}) +O(p^s \log d)\} -|\nabla(Q^s,\R)|/d .
\enq
Lemma~\ref{1overd} and the definition of ``small" give
\beq{Qsiso}
q^s >(1-o(1))p^sd/\log^2d.
\enq

\mn
Set $k=\log_d q^s$ and suppose first that $k=d^{o(1)}$.  Then
Lemma~\ref{1overd} gives
$|N(Q^s)| =\gO(q^sd/k)$, implying that either
\beq{Qs1}
q^s=O(\bg k/d)
\enq
or
\beq{Qs2}
|\nabla(Q^s,\R)|\geq |N(Q^s)|-|\U| =\gO(q^sd/k).
\enq
But if \eqref{Qs1} holds then $i_{P^s}\leq q^s/d$ and \eqref{Qsiso}
imply that the positive terms in \eqref{Psterms} are negligible relative to
\eqref{sillybg}. (Note this uses the fact that $x \log (A/x)$ is increasing on $(0, A/e]$.) If, on the other hand, \eqref{Qs1} does not hold then by \eqref{Qs2} those positive terms are dominated by the negative term.

If $k$ is larger, then $\bg\geq q^s$ implies that the first and second terms in
\eqref{Psterms} are (respectively) $O((\bg/d)\log d)$ and (again using \eqref{Qsiso})
$O((\bg/d)\log^3d)$, both of which are dwarfed by the expression in \eqref{sillybg}.

\mn
\underline{$P^l$ terms}:
\beq{Plbd}
\{O(q^l\log^2d/d)\} + s(\log 3-1) -|\nabla(F,\R)|/d
\enq

\mn
Assuming $q^l\neq 0$, we have $\bg\ge q^l\ge \exp_2[\log ^3 d]$,
so the first term in
\eqref{Plbd} is negligible relative to \eqref{sillybg}.
On the other hand, \eqref{Pi2}
and $S\cap \U=\0$ (see \eqref{SandF}) give
\[
|\nabla(F,\R)|/d\ge |\nabla(F,S)|/d\ge (1-1/\log d)s,
\]
so the sum of the last two terms in \eqref{Plbd} is at most
$-(2-\log 3-1/\log d)s$.

\mn
\emph{Summary.}
In the second case of Lemma~\ref{mp}, the above gains from $A$ and $\hat A$ give the
promised bound (or the stronger
$-\gO(g^s + g^l/\log d + (\hat g/d)\log(\hat g/d))$ mentioned earlier).

If we are in the first case, the desired gain comes from $\bar P^s$ and/or $P^s$
(at least one of which must be nonempty since we assume \eqref{except} does not hold;
note $\bg\le d^2/\log d$ implies $\bar P^l=P^l=\0$).
If $\bar P^s\neq\0$ then the gain is at least $(1/2-o(1))\bar q^s =\gO(d)$.
If $P^s\neq \0$, then we note that \eqref{Qs1} is impossible, since $\bg \le d^2/\log d$ and $q^s>d$;
so \eqref{Qs2} holds and we gain $\gO(q^sd/k)=\gO(d)$.

\subsection{Finally}\label{Finally}
We return to the exceptional case \eqref{except}, which we recall:
\beq{except'}
\mbox{$a=\bar p=p=0~ $ and $~ \hat g\leq d^2/\log d.$}
\enq
Notice that if the first part of this holds then we must have
\[
X\cap \hat G\neq \0
\]
(where $X=X_f$), since otherwise $f$ is not ideal (so is not in $\f^*$).
So it is enough to show that for each $x\in [1,\hat g]$,
the number of possibilities
for $f\in \f^*$ satisfying \eqref{except'} and $|X_f\cap \hat G|=x$
is (suitably) small.

To begin (given $x$) we pay
\beq{except1}
\mbox{$\log\C{N/2}{\hat a} +\log (\hat g/d) +\log\C{N/2}{\hat p} + \log\C{\hat g}{x}
< \hat g+\hat q + O(x\log d)$}
\enq
for $\hat A$, $\hat p$, $\hat P$ and $X\cap \hat G$.
We then
assign colors to $\hat A\cup (X\cap \hat G)\cup \hat P$,
noting that these determine the restriction of $f$ to $\hat G\cup N(X\cap \hat G)\cup \hat Q$
(since $u\in \hat G\sm X$ is colored by whichever of $3,4$ is not assigned to its neighbor
in $\hat A$, and similarly for $v\in (N(X\cap \hat G)\sm \hat A)\cup \hat Q$).
Thus, since vertices whose colors are not determined by these choices are good,
the total coloring cost is at most
\[
\hat a +x+\hat p +N-[\hat g +|N(X\cap \hat G)| + \hat q]
= N-[\hat g+\hat q +\gO(xd)].
\]
(For the r.h.s.\ note that $N(X\cap \hat G)\cap \hat Q=\0$ (by the definition of $\hat P$)
and that the bound on $\hat g$ in \eqref{except'}
implies $|N(X\cap \hat G)|=\gO(xd)$ (by Lemma~\ref{1overd}) and ($\hat p\leq$) $\hat a\leq d/\log d$.)
Finally, combining with \eqref{except1} and summing
bounds the number of $f$'s satisfying \eqref{except'} by
\[
\mbox{$\sum_{x\geq 1}2^{N-\gO(xd)} =2^{N-\gO(d)}.$}
\]

\section{More colors}\label{More}

As mentioned in Section~\ref{Intro}, the conjecture of Engbers and Galvin 
applies to a general (fixed) $q$, but for 
$q>4$ predicts less than the actual asymptotics of $C_q(Q_d)$.  Here we just want to observe that for $q\in \{5,6\}$, one may
again \emph{hope} for something like Theorem~\ref{mainthm}.  We first recall the original conjecture:
\begin{conjecture}\label{EGConj}
[\cite{EG}, Conj.\ 6.2]  For each fixed q,
\beq{EGC}
\mbox{$C_q(Q_d) = (1+{\bf 1}_{\{q ~ {\rm odd}\}})
\C{q}{\lfloor q/2\rfloor}(\lfloor q/2\rfloor \lceil q/2\rceil)^{N/2}
\exp[(1+o(1))f(q)]$}
\enq
as $d\ra\infty$, where 
\[
\mbox{$f(q) =\frac{\lceil q/2\rceil}{2\lfloor q/2\rfloor }\left(2-\frac{2}{\lceil q/2\rceil}\right)^d 
+ 
\frac{\lfloor q/2\rfloor}{2\lceil q/2\rceil}\left(2-\frac{2}{\lfloor q/2\rfloor}\right)^d .$}
\]
\end{conjecture}

\nin
(In particular when $q\in \{3,4\}$, $f(q)=1$ and \eqref{EGC} becomes Theorems~\ref{thmG}
and \ref{mainthm}.)
The first two factors on the r.h.s.\ of \eqref{EGC} correspond to a choice of main phase 
and the third to "pure" colorings (those without flaws) with a given main phase.

The final exponential corresponds to the "isolated" flaws---those at distance at least three from other
flaws---admissible in ideal colorings.  For simplicity we say this just for even $q$
(where $f(q)=(2-4/q)^d$).
Here, for a given main phase the number of ideal colorings with exactly $k$ flaws is at most
\[
\mbox{$\C{N}{k}\left(\frac{q}{2}\right)^{N-dk}\left(\frac{q}{2}-1\right)^{dk} 
< \left(\frac{q}{2}\right)^N\frac{1}{k!}\left[2^d\left(1-\frac{2}{q}\right)^d\right]^k 
=\left(\frac{q}{2}\right)^Nf(q)^k/k!$}
\]
Thus the r.h.s.\ of \eqref{EGC}---even without the "$o(1)$"---is an upper bound on the number of
ideal colorings, and Conjecture~\ref{EGConj} says that, for any (fixed) $q$, this value is not
\emph{so} far from the \emph{overall} number of colorings.

When $q\leq 6$ (that is, when $f(q)\ll \sqrt{N}$), the r.h.s.\ of \eqref{EGC}
without the $o(1)$ is asymptotic to the number of ideal colorings.  We believe that here, as in
Theorem~\ref{mainthm}, the number of \emph{non}-ideal colorings is minor:
\begin{conjecture}\label{34Conj}
For $q\in \{5,6\}$,
\[
\mbox{$C_q(Q_d) \sim (1+{\bf 1}_{\{q ~ {\rm odd}\}})
\C{q}{\lfloor q/2\rfloor}(\lfloor q/2\rfloor \lceil q/2\rceil)^{N/2}
\exp[f(q)]$}
\]
as $d\ra\infty$.
\end{conjecture}

\nin
(For larger $q$, the expression in \eqref{EGC}---now \emph{with} the $o(1)$, which as usual can be 
negative---is a \emph{lower} bound on the number of ideal colorings, so also on $C_q(Q_d)$.
In these cases we suspect that a negative $o(1)$ is the truth.)



\begin{thebibliography}{99}

\bibitem{Sh}
F.R.K. Chung, P. Frankl, R. Graham and J.B. Shearer,
Some intersection theorems for ordered sets and graphs,
pp. 23-37 in \emph{J. Comb. Theory Ser. A} \textbf{48},
1986.

\bibitem{EG}
J. Engbers and D. Galvin,
H-coloring tori,
pp. 1110-1133 in \emph{J. Comb. Theory Ser. B} \textbf{102},
2012.

\bibitem{G}
D. Galvin,
On homomorphisms from the Hamming cube to Z,
pp. 189-213 in \emph{Israel J. of Math} \textbf{138},
2003.

\bibitem{GS}
D. Galvin,
Independent sets in the discrete hypercube,
arXiv:1901.01991 [math.CO]

\bibitem{BHM} J. Kahn,
Range of cube-indexed random walk,
\emph{Isr.\ J.\ Math.} \textbf{124} (2001), 189-201.


\bibitem{Knuth}
D. Knuth,
The art of computer programming Vol. I,
\emph{Addison Wesley}, London, 1969.

\bibitem{KS}
A.\ D.\ Korshunov and A.\ A.\ Sapozhenko,
The number of binary codes with distance 2,
\emph{Problemy Kibernet.} \textbf{40} (1983), 111-130. (Russian)

\bibitem{KW}
J.\ K\"orner and V.\ Wei,
Odd and even Hamming spheres also have minimum boundary,
\emph{Discrete Math.} \textbf{51} (1984), 147-165.


\bibitem{McE}
R.J. McEliece,
\emph{The Theory of Information and Coding,}
Addison-Wesley, London, 1977

\bibitem{PS}  R.\ Peled and Y.\ Spinka,
Rigidity of proper colorings of $\mathbb Z^d$,
arXiv:1808.03597 [math.PR]

\bibitem{Sap87}
A. A. Sapozhenko,
On the number of connected subsets with given cardinality of the boundary in bipartite graphs,
pp. 42-70 in \emph{Metody Diskret. Analiz.} \textbf{45},
1987.



\end{thebibliography}
\end{document}